\documentclass{amsart}
\usepackage{latexsym}
\usepackage{mathrsfs}
\usepackage{indentfirst}
\usepackage[initials]{amsrefs}
\usepackage{multirow}

\theoremstyle{plain}
  \newtheorem{theorem}{Theorem}[section]
  \newtheorem{lemma}{Lemma}[section]

\theoremstyle{definition}

\theoremstyle{remark}
  \newtheorem{remark}{Remark}[section]

\makeatletter
\newcounter{@twon@}
\makeatother

\makeatletter
\newenvironment{condition}[1]{
  \list{}{
    \setbox\@tempboxa=\hbox{#1}
    \setlength{\labelwidth}{\wd\@tempboxa}
    \advance\labelwidth\labelsep
    \setlength{\leftmargin}{\labelwidth}

  }
}{\endlist}

\makeatother

\newcommand{\defequal}{\ensuremath{\overset{\textrm{def}}{=}}}

\newcommand{\RR}{\ensuremath{\mathbb{R}}}

\newcommand{\NN}{\ensuremath{\mathbb{N}}}

\numberwithin{equation}{section}

\allowdisplaybreaks

\begin{document}


\title[Fixed point theorem and the Collatz conjecture]{Fixed point theorem in metric spaces and its application to the Collatz conjecture}

\author{Toshiharu Kawasaki}
\address{Faculty of Engineering, Tamagawa University, Tokyo 194--8610, Japan; Faculty of Science, Chiba University, Chiba 263--8522, Japan}
\email{toshiharu.kawasaki@nifty.ne.jp}

\date{\today}



\begin{abstract}
In this paper, we show the new fixed point theorem in metric spaces. Furthermore, for this fixed point theorem, we apply to the Collatz conjecture.
\end{abstract}

\maketitle


\section{Introduction}
\label{sec:introduction}


Let $\NN \defequal \{ 1, 2, 3, \ldots \}$ and let $C$ be a mapping from $\NN$ into itself defined by
\begin{eqnarray*}
Cx \defequal \left\{\begin{array}{ll}
                   \frac{1}{2}x, & \textrm{if} \ x \ \textrm{is even}, \\
                   3x+1, & \textrm{if} \ x \ \textrm{is odd}.
                   \end{array}\right.
\end{eqnarray*}
Then the Collatz conjecture is as follows: for any $x \in \NN$, there exists $c(x) \in \NN$ such that $C^{c(x)}x = 1$. Many researchers have attempted to prove this conjecture. A most famous paper is probably \cite{Tao2022} by Tao. In his paper, define $\textup{Col}_{\textup{min}}(N) \defequal \inf\{ N, C(N), C^{2}(N), \ldots \}$ and he showed that
\begin{theorem}
Let $f$ be a function from $\NN$ into $\RR$ with $\lim_{N \to \infty}f(N) = \infty$. Then $\textup{Col}_{\textup{min}}(N) < f(N)$ for allmost all $N \in \NN$ in the sence of logarithmic density.
\end{theorem}
\noindent Using this theorem, we obtain $\textup{Col}_{\textup{min}}(N) < \log\log\log\log N$ for almost all $N \in \NN$. Furthermore, see the papers referenced in his paper.

In this paper, in the section~\ref{sec:fixed-point-theorem}, we show the new fixed point theorem in metric spaces. This fixed point theorem is an extension of result showed in \cite{Kawasaki2016a}. Although these results applies to Banach spaces, see also \cite{Kawasaki2018a} and \cite{Kawasaki2021}.

Define $d(x, y) \defequal |x-y|$ for any $x, y \in \NN$. Then $(\NN, d)$ is a complete metric space. In the section~\ref{sec:applying-to-the-collatz-conjecture}, for this fixed point theorem, we apply to the Collatz conjecture.


\section{Fixed point theorem}
\label{sec:fixed-point-theorem}


Let $(X, d)$ be a metric space with a metric $d$. We consider mappings $\alpha$, $\beta$, $\gamma$, $\delta$, $\varepsilon$, and $\zeta$ from $X\times X$ into $\RR$. A mapping $T$ from $X$ into itself is called an $(\alpha, \beta, \gamma, \delta, \varepsilon, \zeta)$-weighted generalized pseudocontraction if
\begin{eqnarray*}
& & \alpha(x, y)d(Tx, Ty)^{2}+\beta(x, y)d(x, Ty)^{2}+\gamma(x, y)d(Tx, y)^{2}+\delta(x, y)d(x, y)^{2} \\
& & \quad +\varepsilon(x, y)d(x, Tx)^{2}+\zeta(x, y)d(y, Ty)^{2} \\
& & \le 0
\end{eqnarray*}
holds for any $x, y \in X$.

\begin{lemma}
\label{lemma:1}
Let $(X, d)$ be a metric space. Then
\begin{eqnarray*}
2\min\{ \theta, 0 \}(d(x, z)^{2}+d(z, y)^{2}) \le \theta d(x, y)^{2}
\end{eqnarray*}
holds for any $\theta \in \RR$ and for any $x, y, z \in X$.
\end{lemma}
\begin{proof}
By the triangle inequality
\begin{eqnarray*}
|d(x, z)-d(z, y)| \le d(x, y) \le d(x, z)+d(z, y)
\end{eqnarray*}
holds for any $x, y, z \in X$. Squaring each side, we obtain
\begin{eqnarray*}
d(x, z)^{2}-2d(x, z)d(z, y)+d(z, y)^{2} \le d(x, y)^{2} \le d(x, z)^{2}+2d(x, z)d(z, y)+d(z, y)^{2}.
\end{eqnarray*}
Therefore, we obtain
\begin{eqnarray*}
\theta d(x, z)^{2}-2|\theta|d(x, z)d(z, y)+\theta d(z, y)^{2} \le \theta d(x, y)^{2}
\end{eqnarray*}
for any $\theta \in \RR$. Since
\begin{eqnarray*}
& & \theta d(x, z)^{2}-2|\theta|d(x, z)d(z, y)+\theta d(z, y)^{2} \\
& & = |\theta|(d(x, z)-d(z, y))^{2}+(\theta-|\theta|)(d(x, z)^{2}+d(z, y)^{2}) \\
& & \ge (\theta-|\theta|)(d(x, z)^{2}+d(z, y)^{2}) \\
& & = 2\min\{ \theta, 0 \}(d(x, z)^{2}+d(z, y)^{2}),
\end{eqnarray*}
we obtain the desired result.
\end{proof}

\begin{lemma}
\label{lemma:2}
Let $(X, d)$ be a metric space, let $T$ be an $(\alpha, \beta, \gamma, \delta, \varepsilon, \zeta)$-weighted generalized pseudocontraction, let $\lambda$ be a mapping from $X\times X$ into $[0, 1]$, and let
\begin{eqnarray*}
\alpha_{\lambda}(x, y) & \defequal & (1-\lambda(x, y))\alpha(x, y)+\lambda(x, y)\alpha(y, x); \\
\beta_{\lambda}(x, y) & \defequal & (1-\lambda(x, y))\beta(x, y)+\lambda(x, y)\gamma(y, x); \\
\gamma_{\lambda}(x, y) & \defequal & (1-\lambda(x, y))\gamma(x, y)+\lambda(x, y)\beta(y, x); \\
\delta_{\lambda}(x, y) & \defequal & (1-\lambda(x, y))\delta(x, y)+\lambda(x, y)\delta(y, x); \\
\varepsilon_{\lambda}(x, y) & \defequal & (1-\lambda(x, y))\varepsilon(x, y)+\lambda(x, y)\zeta(y, x); \\
\zeta_{\lambda}(x, y) & \defequal & (1-\lambda(x, y))\zeta(x, y)+\lambda(x, y)\varepsilon(y, x).
\end{eqnarray*}
Then $T$ is an $(\alpha_{\lambda}, \beta_{\lambda}, \gamma_{\lambda}, \delta_{\lambda}, \varepsilon_{\lambda}, \zeta_{\lambda})$-weighted generalized pseudocontraction.
\end{lemma}
\begin{proof}
By assumption
\begin{eqnarray*}
& & \alpha(x, y)d(Tx, Ty)^{2}+\beta(x, y)d(x, Ty)^{2}+\gamma(x, y)d(Tx, y)^{2}+\delta(x, y)d(x, y)^{2} \\
& & \quad +\varepsilon(x, y)d(x, Tx)^{2}+\zeta(x, y)d(y, Ty)^{2} \\
& & \le 0
\end{eqnarray*}
holds for any $x, y \in X$. Exchanging $x$ and $y$, we obtain
\begin{eqnarray*}
& & \alpha(y, x)d(Tx, Ty)^{2}+\gamma(y, x)d(x, Ty)^{2}+\beta(y, x)d(Tx, y)^{2}+\delta(y, x)d(x, y)^{2} \\
& & \quad +\zeta(y, x)d(x, Tx)^{2}+\varepsilon(y, x)d(y, Ty)^{2} \\
& & \le 0
\end{eqnarray*}
for any $x, y \in X$. By multiplying the upper inequality by $1-\lambda(x, y)$, multiplying the lower inequality by $\lambda(x, y)$, and adding them together, we obtain the desired result.
\end{proof}

\begin{theorem}
\label{th:1}
Let $(X, d)$ be a metric space and let $T$ be an $(\alpha, \beta, \gamma, \delta, \varepsilon, \zeta)$-weighted generalized pseudocontraction. Suppose that there exists a mapping $\lambda$ from $X\times X$ into $[0, 1]$ such that one of the following holds:
\begin{condition}{0000}
\item[(1)] $\alpha_{\lambda}(x, y)+\zeta_{\lambda}(x, y)+2\min\{ \beta_{\lambda}(x, y), 0 \} > 0$ and $\delta_{\lambda}(x, y)+\varepsilon_{\lambda}(x, y)+2\min\{ \beta_{\lambda}(x, \allowbreak y), \allowbreak 0 \} \ge 0$ for any $x, y \in X$;
\item[(2)] $\alpha_{\lambda}(x, y)+\zeta_{\lambda}(x, y)+2\min\{ \beta_{\lambda}(x, y), 0 \} \ge 0$ and $\delta_{\lambda}(x, y)+\varepsilon_{\lambda}(x, y)+2\min\{ \beta_{\lambda}(x, \allowbreak y), \allowbreak 0 \} > 0$ for any $x, y \in X$;
\item[(3)] $\alpha_{\lambda}(x, y)+\varepsilon_{\lambda}(x, y)+2\min\{ \gamma_{\lambda}(x, y), 0 \} > 0$ and $\delta_{\lambda}(x, y)+\zeta_{\lambda}(x, y)+2\min\{ \gamma_{\lambda}(x, \allowbreak y), \allowbreak 0 \} \ge 0$ for any $x, y \in X$;
\item[(4)] $\alpha_{\lambda}(x, y)+\varepsilon_{\lambda}(x, y)+2\min\{ \gamma_{\lambda}(x, y), 0 \} \ge 0$ and $\delta_{\lambda}(x, y)+\varepsilon_{\lambda}(x, y)+2\min\{ \beta_{\lambda}(x, \allowbreak y), \allowbreak 0 \} > 0$ for any $x, y \in X$;
\item[(5)] there exists $A \in (0, 1)$ such that for any $x, y \in X$, $\alpha_{\lambda}(x, y)+\zeta_{\lambda}(x, y)+2\min\{ \beta_{\lambda}(x, y), 0 \} > 0$ and
\begin{eqnarray*}
-\frac{\delta_{\lambda}(x, y)+\varepsilon_{\lambda}(x, y)+2\min\{ \beta_{\lambda}(x, y), 0 \}}{\alpha_{\lambda}(x, y)+\zeta_{\lambda}(x, y)+2\min\{ \beta_{\lambda}(x, y), 0 \}} \le A,
\end{eqnarray*}
or $\alpha_{\lambda}(y, x)+\varepsilon_{\lambda}(y, x)+2\min\{ \gamma_{\lambda}(y, x), 0 \} > 0$ and
\begin{eqnarray*}
-\frac{\delta_{\lambda}(y, x)+\zeta_{\lambda}(y, x)+2\min\{ \gamma_{\lambda}(y, x), 0 \}}{\alpha_{\lambda}(y, x)+\varepsilon_{\lambda}(y, x)+2\min\{ \gamma_{\lambda}(y, x), 0 \}} \le A.
\end{eqnarray*}
\end{condition}
Then $\{ T^{n}x \mid n \in \NN \}$ is Cauchy for any $x \in X$.
\end{theorem}
\begin{proof}
By Lemma~\ref{lemma:2} $T$ is an $(\alpha_{\lambda}, \beta_{\lambda}, \gamma_{\lambda}, \delta_{\lambda}, \varepsilon_{\lambda}, \zeta_{\lambda})$-weighted generalized pseudocontraction. By Lemma~\ref{lemma:1} we obtain
\begin{eqnarray*}
& & \alpha_{\lambda}(x, y)d(Tx, Ty)^{2}+\gamma_{\lambda}(x, y)d(Tx, y)^{2}+(\delta_{\lambda}(x, y)+2\min\{ \beta_{\lambda}(x, y), 0 \})d(x, y)^{2} \\
& & \quad +\varepsilon_{\lambda}(x, y)d(x, Tx)^{2}+(\zeta_{\lambda}(x, y)+2\min\{ \beta_{\lambda}(x, y), 0 \})d(y, Ty)^{2} \\
& & \le 0, \\
& & \alpha_{\lambda}(x, y)d(Tx, Ty)^{2}+\beta_{\lambda}(x, y)d(x, Ty)^{2}+(\delta_{\lambda}(x, y)+2\min\{ \gamma_{\lambda}(x, y), 0 \})d(x, y)^{2} \\
& & \quad +(\varepsilon_{\lambda}(x, y)+2\min\{ \gamma_{\lambda}d(x, y), 0 \})d(x, Tx)^{2}+\zeta_{\lambda}(x, y)d(y, Ty)^{2} \\
& & \le 0.
\end{eqnarray*}
Replacing $x$ with $T^{n-1}x$ and $y$ with $T^{n}x$ in the inequality above, we obtain
\begin{eqnarray*}
& & (\alpha_{\lambda}(T^{n-1}x, T^{n}x)+\zeta_{\lambda}(T^{n-1}x, T^{n}x)+2\min\{ \beta_{\lambda}(T^{n-1}x, T^{n}x), 0 \})) \\
& & \quad\quad\quad\quad \times d(T^{n}x, T^{n+1}x)^{2} \\
& & \quad +(\delta_{\lambda}(T^{n-1}x, T^{n}x)+\varepsilon_{\lambda}(T^{n-1}x, T^{n}x)+2\min\{ \beta_{\lambda}(T^{n-1}x, T^{n}x), 0 \})) \\
& & \quad\quad\quad\quad \times d(T^{n-1}x, T^{n}x)^{2} \\
& & \le 0.
\end{eqnarray*}
In the case of (1), since
\begin{eqnarray*}
d(T^{n}x, T^{n+1}x) = 0
\end{eqnarray*}
holds for any $n \in \NN$, $\{ T^{n}x \mid n \in \NN \}$ is Cauchy. In the case of (2), since
\begin{eqnarray*}
d(T^{n-1}x, T^{n}x) = 0
\end{eqnarray*}
holds for any $n \in \NN$, $\{ T^{n}x \mid n \in \NN \}$ is Cauchy. Replacing $x$ with $T^{n}x$ and $y$ with $T^{n-1}x$ in the inequality below, we obtain
\begin{eqnarray*}
& & (\alpha_{\lambda}(T^{n}x, T^{n-1}x)+\varepsilon_{\lambda}(T^{n}x, T^{n-1}x)+2\min\{ \gamma_{\lambda}(T^{n}x, T^{n-1}x), 0 \})) \\
& & \quad\quad\quad\quad \times d(T^{n+1}x, T^{n}x)^{2} \\
& & \quad +(\delta_{\lambda}(T^{n}x, T^{n-1}x)+\zeta_{\lambda}(T^{n}x, T^{n-1}x)+2\min\{ \gamma_{\lambda}(T^{n}x, T^{n-1}x), 0 \})) \\
& & \quad\quad\quad\quad \times d(T^{n}x, T^{n-1}x)^{2} \\
& & \le 0.
\end{eqnarray*}
In the case of (3), since
\begin{eqnarray*}
d(T^{n+1}x, T^{n}x) = 0
\end{eqnarray*}
holds for any $n \in \NN$, $\{ T^{n}x \mid n \in \NN \}$ is Cauchy. In the case of (4), since
\begin{eqnarray*}
d(T^{n}x, T^{n-1}x) = 0
\end{eqnarray*}
holds for any $n \in \NN$, $\{ T^{n}x \mid n \in \NN \}$ is Cauchy. In the case of (5), since
\begin{eqnarray*}
& & d(T^{n}x, T^{n+1}x)^{2} \\
& & \le -\frac{\delta_{\lambda}(T^{n-1}x, T^{n}x)+\varepsilon_{\lambda}(T^{n-1}x, T^{n}x)+2\min\{ \beta_{\lambda}(T^{n-1}x, T^{n}x), 0 \}}{\alpha_{\lambda}(T^{n-1}x, T^{n}x)+\zeta_{\lambda}(T^{n-1}x, T^{n}x)+2\min\{ \beta_{\lambda}(T^{n-1}x, T^{n}x), 0 \}} \\
& & \quad\quad\quad\quad \times d(T^{n-1}x, T^{n}x)^{2} \\
& & \le Ad(T^{n-1}x, T^{n}x)^{2} \\
& & \le A^{n}d(x, Tx)^{2}
\end{eqnarray*}
or
\begin{eqnarray*}
& & d(T^{n+1}x, T^{n}x)^{2} \\
& & \le -\frac{\delta_{\lambda}(T^{n}x, T^{n-1}x)+\zeta_{\lambda}(T^{n}x, T^{n-1}x)+2\min\{ \gamma_{\lambda}(T^{n}x, T^{n-1}x), 0 \}}{\alpha_{\lambda}(T^{n}x, T^{n-1}x)+\varepsilon_{\lambda}(T^{n}x, T^{n-1}x)+2\min\{ \gamma_{\lambda}(T^{n}x, T^{n-1}x), 0 \}} \\
& & \quad\quad\quad\quad \times d(T^{n}x, T^{n-1}x)^{2} \\
& & \le Ad(T^{n-1}x, T^{n}x)^{2} \\
& & \le A^{n}d(x, Tx)^{2},
\end{eqnarray*}
we obtain
\begin{eqnarray*}
d(T^{n}x, T^{m}x) & \le & \sum_{k = n}^{m-1}d(T^{k}x, T^{k+1}x) \\
& \le & \sum_{k = n}^{m-1}A^{\frac{k}{2}}d(x, Tx) \\
& \le & \frac{A^{\frac{n}{2}}}{1-A^{\frac{1}{2}}}d(x, Tx)
\end{eqnarray*}
for any $m > n$. Therefore, $\{ T^{n}x \mid n \in \NN \}$ is Cauchy.
\end{proof}

By Theorem~\ref{th:1} we obtain the following directly.

\begin{theorem}
\label{th:2}
Let $(X, d)$ be a complete metric space and let $T$ be an $(\alpha, \beta, \gamma, \delta, \varepsilon, \zeta)$-weighted generalized pseudocontraction. Suppose that there exists a mapping $\lambda$ from $X\times X$ into $[0, 1]$ such that one of the following holds:
\begin{condition}{0000}
\item[(1)] $\alpha_{\lambda}(x, y)+\zeta_{\lambda}(x, y)+2\min\{ \beta_{\lambda}(x, y), 0 \} > 0$ and $\delta_{\lambda}(x, y)+\varepsilon_{\lambda}(x, y)+2\min\{ \beta_{\lambda}(x, \allowbreak y), \allowbreak 0 \} \ge 0$ for any $x, y \in X$;
\item[(2)] $\alpha_{\lambda}(x, y)+\zeta_{\lambda}(x, y)+2\min\{ \beta_{\lambda}(x, y), 0 \} \ge 0$ and $\delta_{\lambda}(x, y)+\varepsilon_{\lambda}(x, y)+2\min\{ \beta_{\lambda}(x, \allowbreak y), \allowbreak 0 \} > 0$ for any $x, y \in X$;
\item[(3)] $\alpha_{\lambda}(x, y)+\varepsilon_{\lambda}(x, y)+2\min\{ \gamma_{\lambda}(x, y), 0 \} > 0$ and $\delta_{\lambda}(x, y)+\zeta_{\lambda}(x, y)+2\min\{ \gamma_{\lambda}(x, \allowbreak y), \allowbreak 0 \} \ge 0$ for any $x, y \in X$;
\item[(4)] $\alpha_{\lambda}(x, y)+\varepsilon_{\lambda}(x, y)+2\min\{ \gamma_{\lambda}(x, y), 0 \} \ge 0$ and $\delta_{\lambda}(x, y)+\varepsilon_{\lambda}(x, y)+2\min\{ \beta_{\lambda}(x, \allowbreak y), \allowbreak 0 \} > 0$ for any $x, y \in X$;
\item[(5)] there exists $A \in (0, 1)$ such that for any $x, y \in X$, $\alpha_{\lambda}(x, y)+\zeta_{\lambda}(x, y)+2\min\{ \beta_{\lambda}(x, y), 0 \} > 0$ and
\begin{eqnarray*}
-\frac{\delta_{\lambda}(x, y)+\varepsilon_{\lambda}(x, y)+2\min\{ \beta_{\lambda}(x, y), 0 \}}{\alpha_{\lambda}(x, y)+\zeta_{\lambda}(x, y)+2\min\{ \beta_{\lambda}(x, y), 0 \}} \le A,
\end{eqnarray*}
or $\alpha_{\lambda}(y, x)+\varepsilon_{\lambda}(y, x)+2\min\{ \gamma_{\lambda}(y, x), 0 \} > 0$ and
\begin{eqnarray*}
-\frac{\delta_{\lambda}(y, x)+\zeta_{\lambda}(y, x)+2\min\{ \gamma_{\lambda}(y, x), 0 \}}{\alpha_{\lambda}(y, x)+\varepsilon_{\lambda}(y, x)+2\min\{ \gamma_{\lambda}(y, x), 0 \}} \le A.
\end{eqnarray*}
\end{condition}
Then $\{ T^{n}x \mid n \in \NN \}$ is convergent to a point in $X$ for any $x \in X$.
\end{theorem}

By Theorem~\ref{th:2} we obtain the following.

\begin{theorem}
\label{th:3}
Let $(X, d)$ be a complete metric space and let $T$ be an $(\alpha, \beta, \gamma, \delta, \varepsilon, \zeta)$-weighted generalized pseudocontraction. Suppose that there exists a mapping $\lambda$ from $X\times X$ into $[0, 1]$ such that one of the following holds:
\begin{condition}{0000}
\item[(1)] $\alpha_{\lambda}(x, y)+\zeta_{\lambda}(x, y)+2\min\{ \beta_{\lambda}(x, y), 0 \} > 0$ and $\delta_{\lambda}(x, y)+\varepsilon_{\lambda}(x, y)+2\min\{ \beta_{\lambda}(x, \allowbreak y), \allowbreak 0 \} \ge 0$ for any $x, y \in X$;
\item[(2)] $\alpha_{\lambda}(x, y)+\zeta_{\lambda}(x, y)+2\min\{ \beta_{\lambda}(x, y), 0 \} \ge 0$ and $\delta_{\lambda}(x, y)+\varepsilon_{\lambda}(x, y)+2\min\{ \beta_{\lambda}(x, \allowbreak y), \allowbreak 0 \} > 0$ for any $x, y \in X$;
\item[(3)] $\alpha_{\lambda}(x, y)+\varepsilon_{\lambda}(x, y)+2\min\{ \gamma_{\lambda}(x, y), 0 \} > 0$ and $\delta_{\lambda}(x, y)+\zeta_{\lambda}(x, y)+2\min\{ \gamma_{\lambda}(x, \allowbreak y), \allowbreak 0 \} \ge 0$ for any $x, y \in X$;
\item[(4)] $\alpha_{\lambda}(x, y)+\varepsilon_{\lambda}(x, y)+2\min\{ \gamma_{\lambda}(x, y), 0 \} \ge 0$ and $\delta_{\lambda}(x, y)+\varepsilon_{\lambda}(x, y)+2\min\{ \beta_{\lambda}(x, \allowbreak y), \allowbreak 0 \} > 0$ for any $x, y \in X$;
\item[(5)] there exist $A \in (0, 1)$ and $B \in (0, \infty)$ such that  for any $x, y \in X$, $\alpha_{\lambda}(x, y)+\zeta_{\lambda}(x, y)+2\min\{ \beta_{\lambda}(x, y), \allowbreak 0 \} > 0$,
\begin{eqnarray*}
-\frac{\delta_{\lambda}(x, y)+\varepsilon_{\lambda}(x, y)+2\min\{ \beta_{\lambda}(x, y), 0 \}}{\alpha_{\lambda}(x, y)+\zeta_{\lambda}(x, y)+2\min\{ \beta_{\lambda}(x, y), 0 \}} \le A,
\end{eqnarray*}
and $\alpha_{\lambda}(x, y)+\beta_{\lambda}(x, y)+\zeta_{\lambda}(x, y) \ge B$, or $\alpha_{\lambda}(y, x)+\varepsilon_{\lambda}(y, x)+2\min\{ \gamma_{\lambda}(y, x), \allowbreak 0 \} > 0$,
\begin{eqnarray*}
-\frac{\delta_{\lambda}(y, x)+\zeta_{\lambda}(y, x)+2\min\{ \gamma_{\lambda}(y, x), 0 \}}{\alpha_{\lambda}(y, x)+\varepsilon_{\lambda}(y, x)+2\min\{ \gamma_{\lambda}(y, x), 0 \}} \le A,
\end{eqnarray*}
and $\alpha_{\lambda}(y, x)+\gamma_{\lambda}(y, x)+\varepsilon_{\lambda}(y, x) \ge B$. Furthermore, there exissts $M \in (0, \infty)$ such that $|\alpha(x, y)| \le M$, $|\beta(x, y)| \le M$, $|\gamma(x, y)| \le M$, $|\delta(x, y)| \le M$, $|\varepsilon(x, y)| \le M$, and $|\zeta(x, y)| \le M$ for any $x, y \in X$.
\end{condition}
Then $T$ has a fixed point $u$, where $u = \lim_{n \to \infty}T^{n}x$ for any $x \in X$.
\end{theorem}
\begin{proof}
In the cases of (1) and  (3), the set of all fixed points of $T$ is equal to $T(X)$ and $\lim_{n \to \infty}T^{n}x = Tx$. In the cases of (2) and (4), the set of all fixed points of $T$ is equal to $X$ and $\lim_{n \to \infty}T^{n}x = x$.

We show in the case of (5). By Theorem~\ref{th:2} $\{ T^{n}x \mid n \in \NN \}$ is convergent to a point $u$ in $X$ for any $x \in \NN$. Replacing $x$ with $T^{n}x$ and $y$ with $u$, and replacing $x$ with $u$ and $y$ with $T^{n}x$, we obtain
\begin{eqnarray*}
& & \alpha_{\lambda}(T^{n}x, u)d(T^{n+1}x, Tu)^{2}+\beta_{\lambda}(T^{n}x, u)d(T^{n}x, Tu)^{2} \\
& & \quad +\gamma_{\lambda}(T^{n}x, u)d(T^{n+1}x, u)^{2}+\delta_{\lambda}(T^{n}x, u)d(T^{n}x, u)^{2} \\
& & \quad +\varepsilon_{\lambda}(T^{n}x, u)d(T^{n}x, T^{n+1}x)^{2}+\zeta_{\lambda}(T^{n}x, u)d(u, Tu)^{2} \\
& & \le 0, \\
& & \alpha_{\lambda}(u, T^{n}x)d(Tu, T^{n+1}x)^{2}+\beta_{\lambda}(u, T^{n}x)d(u, T^{n+1}x)^{2} \\
& & \quad +\gamma_{\lambda}(u, T^{n}x)d(Tu, T^{n}x)^{2}+\delta_{\lambda}(u, T^{n}x)d(u, T^{n}x)^{2} \\
& & \quad +\varepsilon_{\lambda}(u, T^{n}x)d(u, Tu)^{2}+\zeta_{\lambda}(u, T^{n}x)d(T^{n}x, T^{n+1}x)^{2} \\
& & \le 0.
\end{eqnarray*}
In the case where $\alpha_{\lambda}(x, y)+\zeta_{\lambda}(x, y)+2\min\{ \beta_{\lambda}(x, y), \allowbreak 0 \} > 0$,
\begin{eqnarray*}
-\frac{\delta_{\lambda}(x, y)+\varepsilon_{\lambda}(x, y)+2\min\{ \beta_{\lambda}(x, y), 0 \}}{\alpha_{\lambda}(x, y)+\zeta_{\lambda}(x, y)+2\min\{ \beta_{\lambda}(x, y), 0 \}} \le A,
\end{eqnarray*}
and $\alpha_{\lambda}(x, y)+\beta_{\lambda}(x, y)+\zeta_{\lambda}(x, y) \ge B$, for any $\rho \in (0, \infty)$, there exists $N \in \NN$ such that $|d(T^{n}x, Tu)^{2}-d(u, Tu)^{2}| < \rho$, $d(T^{n}x, u)^{2} < \rho$, and $d(T^{n}x, T^{n+1}x)^{2} < \rho$ for any $n > N$. Therefore,
\begin{eqnarray*}
& & \alpha_{\lambda}(T^{n}x, u)d(T^{n+1}x, Tu)^{2}+\beta_{\lambda}(T^{n}x, u)d(T^{n}x, Tu)^{2} \\
& & \quad +\gamma_{\lambda}(T^{n}x, u)d(T^{n+1}x, u)^{2}+\delta_{\lambda}(T^{n}x, u)d(T^{n}x, u)^{2} \\
& & \quad +\varepsilon_{\lambda}(T^{n}x, u)d(T^{n}x, T^{n+1}x)^{2}+\zeta_{\lambda}(T^{n}x, u)d(u, Tu)^{2} \\
& & > \alpha_{\lambda}(T^{n}x, u)d(u, Tu)^{2}-M\rho+\beta_{\lambda}(T^{n}x, u)d(u, Tu)^{2}-M\rho \\
& & \quad -M\rho-M\rho \\
& & \quad -M\rho+\zeta_{\lambda}(T^{n}x, u)d(u, Tu)^{2} \\
& & \ge Bd(u, Tu)^{2}-5M\rho.
\end{eqnarray*}
Since $\rho$ is arbitrary, $d(u, Tu) \le 0$, that is, $u$ is a fixed point of $T$. In the case where $\alpha_{\lambda}(y, x)+\varepsilon_{\lambda}(y, x)+2\min\{ \gamma_{\lambda}(y, x), \allowbreak 0 \} > 0$,
\begin{eqnarray*}
-\frac{\delta_{\lambda}(y, x)+\zeta_{\lambda}(y, x)+2\min\{ \gamma_{\lambda}(y, x), 0 \}}{\alpha_{\lambda}(y, x)+\varepsilon_{\lambda}(y, x)+2\min\{ \gamma_{\lambda}(y, x), 0 \}} \le A,
\end{eqnarray*}
and $\alpha_{\lambda}(y, x)+\gamma_{\lambda}(y, x)+\varepsilon_{\lambda}(y, x) \ge B$, for any $\rho \in (0, \infty)$, there exists $N \in \NN$ such that $|d(Tu, T^{n}x)^{2}-d(Tu, u)^{2}| < \rho$, $d(u, T^{n}x)^{2} < \rho$, and $d(T^{n+1}x, T^{n}x)^{2} < \rho$ for any $n > N$. Therefore,
\begin{eqnarray*}
& & \alpha_{\lambda}(u, T^{n}x)d(Tu, T^{n+1}x)^{2}+\beta_{\lambda}(u, T^{n}x)d(u, T^{n+1}x)^{2} \\
& & \quad +\gamma_{\lambda}(u, T^{n}x)d(Tu, T^{n}x)^{2}+\delta_{\lambda}(u, T^{n}x)d(u, T^{n}x)^{2} \\
& & \quad +\varepsilon_{\lambda}(u, T^{n}x)d(u, Tu)^{2}+\zeta_{\lambda}(u, T^{n}x)d(T^{n}x, T^{n+1}x)^{2} \\
& & > \alpha_{\lambda}(u, T^{n}u)d(Tu, u)^{2}-M\rho-M\rho \\
& & \quad +\gamma_{\lambda}(u, T^{n}x)d(Tu, u)^{2}-M\rho-M\rho \\
& & \quad +\varepsilon_{\lambda}(u, T^{n}x)d(u, Tu)^{2}-M\rho \\
& & \ge Bd(Tu, u)^{2}-5M\rho.
\end{eqnarray*}
Since $\rho$ is arbitrary, $d(Tu, u) \le 0$, that is, $u$ is a fixed point of $T$.
\end{proof}


\section{Applying to the Collatz conjecture}
\label{sec:applying-to-the-collatz-conjecture}


Let $C$ be a mapping from $\NN$ into itself defined by
\begin{eqnarray*}
Cx \defequal \left\{\begin{array}{ll}
                   \frac{1}{2}x, & \textrm{if} \ x \ \textrm{is even}, \\
                   3x+1, & \textrm{if} \ x \ \textrm{is odd}.
                   \end{array}\right.
\end{eqnarray*}
Then the Collatz conjecture is as follows: for any $x \in \NN$, there exists $c(x) \in \NN$ such that $C^{c(x)}x = 1$.

Clearly, $c(1) = 3$. Since $3x+1$ is even whenever $x$ is odd, $C^{2}x = \frac{1}{2}(3x+1)$ for any odd number $x$. Let $T$ be a mapping from $\NN$ into itself defined by
\begin{eqnarray*}
Tx \defequal \left\{\begin{array}{ll}
                  C^{3}x = 1, & \textrm{if} \ x = 1, \\
                  Cx = \frac{1}{2}x, & \textrm{if} \ x \ \textrm{is even}, \\
                  C^{2}x = \frac{1}{2}(3x+1), & \textrm{if} \ x \ \textrm{is odd and} \ x \ge 3.
                  \end{array}\right.
\end{eqnarray*}
Then, if we can prove that for any $x \in \NN$, there exists $t(x) \in \NN$ such that $T^{t(x)}x = 1$, then the Collatz conjecture is true.

Define $d(x, y) \defequal |x-y|$ for any $x, y \in \NN$. Then $(\NN, d)$ is a complete metric space.

\begin{theorem}
\label{th:4}
Let
\begin{eqnarray*}
\alpha(x, y) & \defequal & \left\{\begin{array}{cl}
                                    1, & \textrm{if} \ x = 1 \ \textrm{and} \ y = 1, \\
                                    1, & \textrm{if} \ x = 1 \ \textrm{and} \ y \ \textrm{is even}, \\
                                    0, & \textrm{if} \ x = 1, y \ \textrm{is odd, and} \ y \ge 3, \\
                                    1, & \textrm{if} \ x \ \textrm{is even and} \ y = 1, \\
                                    1, & \textrm{if} \ x \ \textrm{is even and} \ y \ \textrm{is even}, \\
                                    0, & \textrm{if} \ x \ \textrm{is even}, y \ \textrm{is odd, and} \ y \ge 3, \\
                                    1, & \textrm{if} \ x \ \textrm{is odd}, \ x \ge 3, \ \textrm{and} \ y = 1, \\
                                    0, & \textrm{if} \ x \ \textrm{is odd}, \ x \ge 3, \ \textrm{and} \ y \ \textrm{is even}, \\
                                    2, & \textrm{if} \ x \ \textrm{is odd}, x \ge 3, y \ \textrm{is odd, and} \ y \ge 3;
                                    \end{array}\right. \\
\beta(x, y) & \defequal & \left\{\begin{array}{cl}
                                    0, & \textrm{if} \ x = 1 \ \textrm{and} \ y = 1, \\
                                    0, & \textrm{if} \ x = 1 \ \textrm{and} \ y \ \textrm{is even}, \\
                                    0, & \textrm{if} \ x = 1, y \ \textrm{is odd, and} \ y \ge 3, \\
                                    0, & \textrm{if} \ x \ \textrm{is even and} \ y = 1, \\
                                    0, & \textrm{if} \ x \ \textrm{is even and} \ y \ \textrm{is even}, \\
                                    0, & \textrm{if} \ x \ \textrm{is even}, y \ \textrm{is odd, and} \ y \ge 3, \\
                                    0, & \textrm{if} \ x \ \textrm{is odd}, \ x \ge 3, \ \textrm{and} \ y = 1, \\
                                    -2, & \textrm{if} \ x \ \textrm{is odd}, \ x \ge 3, \ \textrm{and} \ y \ \textrm{is even}, \\
                                    \beta_{0}(x, y), & \textrm{if} \ x \ \textrm{is odd}, x \ge 3, y \ \textrm{is odd, and} \ y \ge 3;
                                   \end{array}\right. \\
\gamma(x, y) & \defequal & \left\{\begin{array}{cl}
                                    0, & \textrm{if} \ x = 1 \ \textrm{and} \ y = 1, \\
                                    0, & \textrm{if} \ x = 1 \ \textrm{and} \ y \ \textrm{is even}, \\
                                    0, & \textrm{if} \ x = 1, y \ \textrm{is odd, and} \ y \ge 3, \\
                                    1, & \textrm{if} \ x \ \textrm{is even and} \ y = 1, \\
                                    -1, & \textrm{if} \ x \ \textrm{is even and} \ y \ \textrm{is even}, \\
                                    -2, & \textrm{if} \ x \ \textrm{is even}, y \ \textrm{is odd, and} \ y \ge 3, \\
                                    -1, & \textrm{if} \ x \ \textrm{is odd}, \ x \ge 3, \ \textrm{and} \ y = 1, \\
                                    0, & \textrm{if} \ x \ \textrm{is odd}, \ x \ge 3, \ \textrm{and} \ y \ \textrm{is even}, \\
                                    -\beta_{0}(x, y), & \textrm{if} \ x \ \textrm{is odd}, x \ge 3, y \ \textrm{is odd, and} \ y \ge 3;
                                      \end{array}\right. \\
\delta(x, y) & \defequal & \left\{\begin{array}{cl}
                                    0, & \textrm{if} \ x = 1 \ \textrm{and} \ y = 1, \\
                                    -1, & \textrm{if} \ x = 1 \ \textrm{and} \ y \ \textrm{is even}, \\
                                    -2, & \textrm{if} \ x = 1, y \ \textrm{is odd, and} \ y \ge 3, \\
                                    -1, & \textrm{if} \ x \ \textrm{is even and} \ y = 1, \\
                                    0, & \textrm{if} \ x \ \textrm{is even and} \ y \ \textrm{is even}, \\
                                    1, & \textrm{if} \ x \ \textrm{is even}, y \ \textrm{is odd, and} \ y \ge 3, \\
                                    -1, & \textrm{if} \ x \ \textrm{is odd}, \ x \ge 3, \ \textrm{and} \ y = 1, \\
                                    1, & \textrm{if} \ x \ \textrm{is odd}, \ x \ge 3, \ \textrm{and} \ y \ \textrm{is even}, \\
                                    \delta_{0}(x, y), & \textrm{if} \ x \ \textrm{is odd}, x \ge 3, y \ \textrm{is odd, and} \ y \ge 3;
                                   \end{array}\right. \\
\varepsilon(x, y) & \defequal & \left\{\begin{array}{cl}
                                    -1, & \textrm{if} \ x = 1 \ \textrm{and} \ y = 1, \\
                                    0, & \textrm{if} \ x = 1 \ \textrm{and} \ y \ \textrm{is even}, \\
                                    1, & \textrm{if} \ x = 1, y \ \textrm{is odd, and} \ y \ge 3, \\
                                    0, & \textrm{if} \ x \ \textrm{is even and} \ y = 1, \\
                                    -1, & \textrm{if} \ x \ \textrm{is even and} \ y \ \textrm{is even}, \\
                                    -2, & \textrm{if} \ x \ \textrm{is even}, y \ \textrm{is odd, and} \ y \ge 3, \\
                                    0, & \textrm{if} \ x \ \textrm{is odd}, \ x \ge 3, \ \textrm{and} \ y = 1, \\
                                    2, & \textrm{if} \ x \ \textrm{is odd}, \ x \ge 3, \ \textrm{and} \ y \ \textrm{is even}, \\
                                    \varepsilon_{0}(x, y), & \textrm{if} \ x \ \textrm{is odd}, x \ge 3, y \ \textrm{is odd, and} \ y \ge 3;
                                          \end{array}\right. \\
\zeta(x, y) & \defequal & \left\{\begin{array}{cl}
                                    1, & \textrm{if} \ x = 1 \ \textrm{and} \ y = 1, \\
                                    1, & \textrm{if} \ x = 1 \ \textrm{and} \ y \ \textrm{is even}, \\
                                    2, & \textrm{if} \ x = 1, y \ \textrm{is odd, and} \ y \ge 3, \\
                                    1, & \textrm{if} \ x \ \textrm{is even and} \ y = 1, \\
                                    1, & \textrm{if} \ x \ \textrm{is even and} \ y \ \textrm{is even}, \\
                                    2, & \textrm{if} \ x \ \textrm{is even}, y \ \textrm{is odd, and} \ y \ge 3, \\
                                    1, & \textrm{if} \ x \ \textrm{is odd}, \ x \ge 3, \ \textrm{and} \ y = 1, \\
                                    -2, & \textrm{if} \ x \ \textrm{is odd}, \ x \ge 3, \ \textrm{and} \ y \ \textrm{is even}, \\
                                    \zeta_{0}(x, y) & \textrm{if} \ x \ \textrm{is odd}, x \ge 3, y \ \textrm{is odd, and} \ y \ge 3;
                                   \end{array}\right.
\end{eqnarray*}
for any $x, y \in \NN$, where
\begin{eqnarray*}
\beta_{0}(x, y) & \defequal & \left\{\begin{array}{cl}
                                       -2, & \textrm{if} \ k-\ell \le -2, \\
                                       k-\ell, & \textrm{if} \ |k-\ell| \le 1, \\
                                       2, & \textrm{if} \ k-\ell \ge 2;
                                       \end{array}\right. \\
\delta_{0}(x, y) & \defequal & \left\{\begin{array}{cl}
                                       -2, & \textrm{if} \ \begin{array}{l}
                                                                k-\ell \le -2 \ \textrm{and} \ 11k-10\ell+1 \le 0, \\
                                                                \textrm{or} \\
                                                                k-\ell \ge 2 \ \textrm{and} \ -10k+11\ell+1 \le 0,
                                                                \end{array} \\
                                       -1, & \textrm{otherwise};
                                       \end{array}\right. \\
\varepsilon_{0}(x, y) & \defequal & \left\{\begin{array}{cl}
                                              2, & \textrm{if} \ k-\ell \le -2 \ \textrm{and} \ 11k-10\ell+1 \le 0, \\
                                              0, & \textrm{otherwise};
                                              \end{array}\right. \\
\zeta_{0}(x, y) & \defequal & \left\{\begin{array}{cl}
                                       2, & \textrm{if} \ k-\ell \ge 2 \ \textrm{and} \ -10k+11\ell+1 \le 0, \\
                                       0, & \textrm{otherwise};
                                       \end{array}\right. \\
\end{eqnarray*}
for any $x = 2k+1 \ (k \in \NN)$ and for any $y = 2\ell+1 \ (\ell \in \NN)$. Then $T$ is an $(\alpha, \beta, \gamma, \delta, \varepsilon, \zeta)$-weighted generalized pseudocontraction.
\end{theorem}
\begin{proof}
In the case where $x = 1$ and $y = 1$: since $Tx = Ty = 1$, we obtain
\begin{eqnarray*}
& & \alpha(1, 1)d(1, 1)^{2}+\beta(1, 1)d(1, 1)^{2}+\gamma(1, 1)d(1, 1)^{2}+\delta(1, 1)d(1, 1)^{2} \\
& & \quad +\varepsilon(1, 1)d(1, 1)^{2}+\zeta(1, 1)d(1, 1)^{2} \\
& & = 1\times 0^{2}+0\times 0^{2}+0\times 0^{2}+0\times 0^{2} \\
& & \quad +(-1)\times 0^{2}+1\times 0^{2} \\
& & = 0.
\end{eqnarray*}

In the case where $x = 1$ and $y$ is even: let $y = 2\ell \ (\ell \in \NN)$. Since $Tx = 1$ and $Ty = \ell$, we obtain
\begin{eqnarray*}
& & \alpha(1, 2\ell)d(1, \ell)^{2}+\beta(1, 2\ell)d(1, \ell)^{2}+\gamma(1, 2\ell)d(1, 2\ell)^{2}+\delta(1, 2\ell)d(1, 2\ell)^{2} \\
& & \quad +\varepsilon(1, 2\ell)d(1, 1)^{2}+\zeta(1, 2\ell)d(2\ell, \ell)^{2} \\
& & = 1\times(1-\ell)^{2}+0\times(1-\ell)^{2}+0\times(1-2\ell)^{2}+(-1)\times(1-2\ell)^{2} \\
& & \quad +0\times 0^{2}+1\times\ell^{2} \\
& & = -2\ell^{2}+2\ell \\
& & \le 0.
\end{eqnarray*}

In the case where $x$ is even and $y = 1$: let $x = 2k \ (k \in \NN)$. Since $Tx = k$ and $Ty = 1$, we obtain
\begin{eqnarray*}
& & \alpha(2k, 1)d(k, 1)^{2}+\beta(2k, 1)d(2k, 1)^{2}+\gamma(2k, 1)d(k, 1)^{2}+\delta(2k, 1)d(2k, 1)^{2} \\
& & \quad +\varepsilon(2k, 1)d(2k, k)^{2}+\zeta(2k, 1)d(1, 1)^{2} \\
& & = 1\times(k-1)^{2}+0\times(2k-1)^{2}+1\times(k-1)^{2}+(-1)\times(2k-1)^{2} \\
& & \quad +0\times k^{2}+1\times 0^2 \\
& & = -2k^{2}+1 \\
& & \le -1.
\end{eqnarray*}

In the case where $x$ is even and $y$ is even: let $x = 2k \ (k \in \NN)$ and let $y = 2\ell \ (\ell \in \NN)$. Since $Tx = k$ and $Ty = \ell$, we obtain
\begin{eqnarray*}
& & \alpha(2k, 2\ell)d(k, \ell)^{2}+\beta(2k, 2\ell)d(2k, \ell)^{2}+\gamma(2k, 2\ell)d(k, 2\ell)^{2}+\delta(2k, 2\ell)d(2k, 2\ell)^{2} \\
& & \quad +\varepsilon(2k, 2\ell)d(2k, k)^{2}+\zeta(2k, 2\ell)d(2\ell, \ell)^{2} \\
& & = 1\times(k-\ell)^{2}+0\times(2k-\ell)^{2}+(-1)\times(k-2\ell)^{2}+0\times(2k-2\ell)^{2} \\
& & \quad +(-1)\times k^{2}+1\times\ell^2 \\
& & = -k^{2}+2k\ell-2\ell^{2} \\
& & \le -1.
\end{eqnarray*}

In the case where $x = 1$, $y$ is odd, and $y \ge 3$: let $y = 2\ell+1 \ (\ell \in \NN)$. Since $Tx = 1$ and $Ty = 3\ell+2$, we obtain
\begin{eqnarray*}
& & \alpha(1, 2\ell+1)d(1, 3\ell+2)^{2}+\beta(1, 2\ell+1)d(1, 3\ell+2)^{2}+\gamma(1, 2\ell+1)d(1, 2\ell+1)^{2} \\
& & \quad +\delta(1, 2\ell+1)d(1, 2\ell+1)^{2} \\
& & \quad +\varepsilon(1, 2\ell+1)d(1, 1)^{2}+\zeta(1, 2\ell+1)d(2\ell+1, 3\ell+2)^{2} \\
& & = 0\times(3\ell+1)^{2}+0\times(3\ell+1)^{2}+0\times(2\ell)^{2}+(-2)\times(2\ell)^{2} \\
& & \quad +1\times 0^{2}+2\times(\ell+1)^{2} \\
& & = -6\ell^{2}+4\ell+2 \\
& & \le 0.
\end{eqnarray*}

In the case where $x$ is odd, $x \ge 3$, and $y = 1$: let $x = 2k+1 \ (k \in \NN)$. Since $Tx = 3k+2$ and $Ty = 1$, we obtain
\begin{eqnarray*}
& & \alpha(2k+1, 1)d(3k+2, 1)^{2}+\beta(2k+1, 1)d(2k+1, 1)^{2}+\gamma(2k+1, 1)d(3k+2, 1)^{2} \\
& & \quad +\delta(2k+1, 1)d(2k+1, 1)^{2} \\
& & \quad +\varepsilon(2k+1, 1)d(2k+1, 3k+2)^{2}+\zeta(2k+1, 1)d(1, 1)^{2} \\
& & = 1\times(3k+1)^{2}+0\times(2k)^{2}+(-1)\times(3k+1)^{2}+(-1)\times(2k)^{2} \\
& & \quad +0\times(k+1)^{2}+1\times 0^{2} \\
& & = -4k^{2} \\
& & \le -4.
\end{eqnarray*}

In the case where $x$ is even, $y$ is odd, and $y \ge 3$: let $x = 2k \ (k \in \NN)$ and let $y = 2\ell+1 \ (\ell \in \NN)$. Since $Tx = k$ and $Ty = 3\ell+2$, we obtain
\begin{eqnarray*}
& & \alpha(2k, 2\ell+1)d(k, 3\ell+2)^{2}+\beta(2k, 2\ell+1)d(2k, 3\ell+2)^{2} \\
& & \quad +\gamma(2k, 2\ell+1)d(k, 2\ell+1)^{2}+\delta(2k, 2\ell+1)d(2k, 2\ell+1)^{2} \\
& & \quad +\varepsilon(2k, 2\ell+1)d(2k, k)^{2}+\zeta(2k, 2\ell+1)d(2\ell+1, 3\ell+2)^{2} \\
& & = 0\times(k-3\ell-2)^{2}+0\times(2k-3\ell-2)^{2}+(-2)\times(k-2\ell-1)^{2} \\
& & \quad +1\times(2k-2\ell-1)^{2}+(-2)\times k^{2}+2\times(\ell+1)^2 \\
& & = -2\ell^{2}+1 \\
& & \le -1.
\end{eqnarray*}

In the case where $x$ is odd, $x \ge 3$, and $y$ is even: let $x = 2k+1 \ (k \in \NN)$ and $y = 2\ell \ (\ell \in \NN)$. Since $Tx = 3k+2$ and $Ty = \ell$, we obtain
\begin{eqnarray*}
& & \alpha(2k+1, 2\ell)d(3k+2, \ell)^{2}+\beta(2k+1, 2\ell)d(2k+1, \ell)^{2} \\
& & \quad +\gamma(2k+1, 2\ell)d(3k+2, 2\ell)^{2}+\delta(2k+1, 2\ell)d(2k+1, 2\ell)^{2} \\
& & \quad +\varepsilon(2k+1, 2\ell)d(2k+1, 3k+2)^{2}+\zeta(2k+1, 2\ell)d(2\ell, \ell)^{2} \\
& & = 0\times(3k-\ell+2)^{2}+(-2)\times(2k-\ell+1)^{2}+0\times(3k-2\ell+2)^{2} \\
& & \quad +1\times(2k-2\ell+1)^{2}+2\times(k+1)^{2}+(-2)\times\ell^2 \\
& & = -2k^{2}+1 \\
& & \le -1.
\end{eqnarray*}

In the case where $x$ is odd, $x \ge 3$, $y$ is odd, and $y \ge 3$: let $x = 2k+1 \ (k \in \NN)$ and $y = 2\ell+1 \ (\ell \in \NN)$. Since $Tx = 3k+2$ and $Ty = 3\ell+2$, we obtain
\begin{eqnarray*}
& & \alpha(2k+1, 2\ell+1)d(3k+2, 3\ell+2)^{2}+\beta(2k+1, 2\ell+1)d(2k+1, 3\ell+2)^{2} \\
& & \quad +\gamma(2k+1, 2\ell+1)d(3k+2, 2\ell+1)^{2}+\delta(2k+1, 2\ell+1)d(2k+1, 2\ell+1)^{2} \\
& & \quad +\varepsilon(2k+1, 2\ell+1)d(2k+1, 3k+2)^{2}+\zeta(2k+1, 2\ell+1)d(2\ell+1, 3\ell+2)^{2} \\
& & = 2\times(3k-3\ell)^{2}+\beta_{0}(2k+1, 2\ell+1)\times(2k-3\ell-1)^{2} \\
& & \quad -\beta_{0}(2k+1, 2\ell+1)\times(3k-2\ell+1)^{2}+\delta_{0}(2k+1, 2\ell+1)\times(2k-2\ell)^{2} \\
& & \quad +\varepsilon_{0}(2k+1, 2\ell+1)\times(k+1)^{2}+\zeta_{0}(2k+1, 2\ell+1)\times(\ell+1)^{2} \\
& & = (18+4\delta_{0}(2k+1, 2\ell+1))(k-\ell)^{2}-5\beta_{0}(2k+1, 2\ell+1)(k-\ell)(k+\ell+2) \\
& & \quad +\varepsilon_{0}(2k+1, 2\ell+1)(k+1)^{2}+\zeta_{0}(2k+1, 2\ell+1)(\ell+1)^{2}.
\end{eqnarray*}
In the case of $k-\ell \le -2$ and $11k-10\ell+1 \le 0$: since $\beta_{0}(2k+1, 2\ell+1) = -2$, $\delta_{0}(2k+1, 2\ell+1) = -2$, $\varepsilon_{0}(2k+1, 2\ell+1) = 2$ and $\zeta_{0}(2k+1, 2\ell+1) = 0$, we obtain
\begin{eqnarray*}
& & (18+4\delta_{0}(2k+1, 2\ell+1))(k-\ell)^{2}-5\beta_{0}(2k+1, 2\ell+1)(k-\ell)(k+\ell+2) \\
& & \quad +\varepsilon_{0}(2k+1, 2\ell+1)(k+1)^{2}+\zeta_{0}(2k+1, 2\ell+1)(\ell+1)^{2} \\
& & = 10(k-\ell)^{2}+10(k-\ell)(k+\ell+2)+2(k+1)^{2} \\
& & = 2(k+1)(11k-10\ell+1) \\
& & \le 0.
\end{eqnarray*}
In the case of $k-\ell \le -2$ and $11k-10\ell+1 \ge 1$: since $\beta_{0}(2k+1, 2\ell+1) = -2$, $\delta_{0}(2k+1, 2\ell+1) = -1$, $\varepsilon_{0}(2k+1, 2\ell+1) = \zeta_{0}(2k+1, 2\ell+1) = 0$, we obtain
\begin{eqnarray*}
& & (18+4\delta_{0}(2k+1, 2\ell+1))(k-\ell)^{2}-5\beta_{0}(2k+1, 2\ell+1)(k-\ell)(k+\ell+2) \\
& & \quad +\varepsilon_{0}(2k+1, 2\ell+1)(k+1)^{2}+\zeta_{0}(2k+1, 2\ell+1)(\ell+1)^{2} \\
& & = 14(k-\ell)^{2}+10(k-\ell)(k+\ell+2) \\
& & = 4(k-\ell)(6k-\ell+5).
\end{eqnarray*}
Since
\begin{eqnarray*}
5(-k+\ell-2)+(11k-10\ell) = 6k-5\ell-10 \ge 0
\end{eqnarray*}
and $6k-\ell+5 > 6k-5\ell-10$, we obtain
\begin{eqnarray*}
4(k-\ell)(6k-\ell+5) \le -8.
\end{eqnarray*}
In the case of $k-\ell \ge 2$ and $-10k+11\ell+1 \le 0$: since $\beta_{0}(2k+1, 2\ell+1) = 2$, $\delta_{0}(2k+1, 2\ell+1) = -2$, $\varepsilon_{0}(2k+1, 2\ell+1) = 0$ and $\zeta_{0}(2k+1, 2\ell+1) = 2$, we obtain
\begin{eqnarray*}
& & (18+4\delta_{0}(2k+1, 2\ell+1))(k-\ell)^{2}-5\beta_{0}(2k+1, 2\ell+1)(k-\ell)(k+\ell+2) \\
& & \quad +\varepsilon_{0}(2k+1, 2\ell+1)(k+1)^{2}+\zeta_{0}(2k+1, 2\ell+1)(\ell+1)^{2} \\
& & = 10(k-\ell)^{2}-10(k-\ell)(k+\ell+2)+2(\ell+1)^{2} \\
& & = 2(\ell+1)(-10k+11\ell+1) \\
& & \le 0.
\end{eqnarray*}
In the case of $k-\ell \ge 2$ and $-10k+11\ell+1 \ge 1$: since $\beta_{0}(2k+1, 2\ell+1) = 2$, $\delta_{0}(2k+1, 2\ell+1) = -1$, $\varepsilon_{0}(2k+1, 2\ell+1) = \zeta_{0}(2k+1, 2\ell+1) = 0$, we obtain
\begin{eqnarray*}
& & (18+4\delta_{0}(2k+1, 2\ell+1))(k-\ell)^{2}-5\beta_{0}(2k+1, 2\ell+1)(k-\ell)(k+\ell+2) \\
& & \quad +\varepsilon_{0}(2k+1, 2\ell+1)(k+1)^{2}+\zeta_{0}(2k+1, 2\ell+1)(\ell+1)^{2} \\
& & = 14(k-\ell)^{2}-10(k-\ell)(k+\ell+2) \\
& & = 4(k-\ell)(k-6\ell-5).
\end{eqnarray*}
Since
\begin{eqnarray*}
5(-k+\ell-2)+(10k-11\ell) = 5k-6\ell-10 \le 0
\end{eqnarray*}
and $k-6\ell-5 < 5k-6\ell-10$, we obtain
\begin{eqnarray*}
4(k-\ell)(k-6\ell-5) \le -8.
\end{eqnarray*}
In other cases: since $\beta_{0}(2k+1, 2\ell+1) = k-\ell$, $\delta_{0}(2k+1, 2\ell+1) = -1$, $\varepsilon_{0}(2k+1, 2\ell+1) = \zeta_{0}(2k+1, 2\ell+1) = 0$, we obtain
\begin{eqnarray*}
& & (18+4\delta_{0}(2k+1, 2\ell+1))(k-\ell)^{2}-5\beta_{0}(2k+1, 2\ell+1)(k-\ell)(k+\ell+2) \\
& & \quad +\varepsilon_{0}(2k+1, 2\ell+1)(k+1)^{2}+\zeta_{0}(2k+1, 2\ell+1)(\ell+1)^{2} \\
& & = 14(k-\ell)^{2}-5(k-\ell)^{2}(k+\ell+2) \\
& & = (k-\ell)^{2}(4-5(k+\ell)) \\
& & \le 0.
\end{eqnarray*}

From the above, $T$ is an $(\alpha, \beta, \gamma, \delta, \varepsilon, \zeta)$-weighted generalized pseudocontraction.
\end{proof}

\begin{remark}
By Theorem~\ref{th:4} $T$ is an $(\alpha, \beta, \gamma, \delta, \varepsilon, \zeta)$-weighted generalized pseudocontraction. Let $\lambda$ be a mapping from $X\times X$ into $[0, 1]$ defined by $\lambda(x, y) = 1$, let $A = \frac{1}{2}$, let $B = 2$, and let $M = 2$.

Then, do $\alpha_{\lambda}$, $\beta_{\lambda}$, $\gamma_{\lambda}$, $\delta_{\lambda}$, $\varepsilon_{\lambda}$, and $\zeta_{\lambda}$ satisfy (5) in Theorem~\ref{th:3}?

Clearly, $|\alpha_{\lambda}(x, y)| \le M$, $|\beta_{\lambda}(x, y)| \le M$, $|\gamma_{\lambda}(x, y)| \le M$, $|\delta_{\lambda}(x, y)| \le M$, $|\varepsilon_{\lambda}(x, y)| \le M$, and $|\zeta_{\lambda}(x, y)| \le M$.

In the cases where $x = 1$ and $y = 1$, $x = 1$ and $y$ is even, $x$ is even and $y = 1$, $x = 1$ and $y$ is odd and $y \ge 3$, $x$ is odd and $x \ge 3$ and $y = 1$, $x$ is even and $y$ is even, or $x$ is even and $y$ is odd and $y \ge 3$,
\begin{eqnarray*}
\alpha_{\lambda}(x, y)+\zeta_{\lambda}(x, y)+2\min\{ \beta_{\lambda}(x, y), 0 \} & = & \alpha(x, y)+\zeta(x, y)+2\min\{ \beta(x, y), 0 \} \\
& = & 2, \\
\delta_{\lambda}(x, y)+\varepsilon_{\lambda}(x, y)+2\min\{ \beta_{\lambda}(x, y), 0 \} & = & \delta(x, y)+\varepsilon(x, y)+2\min\{ \beta(x, y), 0 \} \\
& = & -1, \\
\alpha_{\lambda}(x, y)+\beta_{\lambda}(x, y)+\zeta_{\lambda}(x, y) & = & \alpha(x, y)+\beta(x, y)+\zeta(x, y) \\
& = & 2 = B
\end{eqnarray*}
and
\begin{eqnarray*}
-\frac{\delta_{\lambda}(x, y)+\varepsilon_{\lambda}(x, y)+2\min\{ \beta_{\lambda}(x, y), 0 \}}{\alpha_{\lambda}(x, y)+\zeta_{\lambda}(x, y)+2\min\{ \beta_{\lambda}(x, y), 0 \}} = \frac{1}{2} = A.
\end{eqnarray*}
In the case where $x$ is odd, $x \ge 3$, $y$ is odd, and $y \ge 3$: if $x \ge y$, then since  the possible combinations of $\alpha(x, y)$, $\beta(x, y)$, $\gamma(x, y)$, $\delta(x, y)$, $\varepsilon(x, y)$, and $\zeta(x, y)$ are $(2, 2, -2, -2, 0, 2)$, $(2, 2, -2, -1, 0, 0)$, $(2, 1, -1, -1, 0, 0)$, and $(2, 0, 0, -1, 0, 0)$,
\begin{eqnarray*}
& & \alpha_{\lambda}(x, y)+\zeta_{\lambda}(x, y)+2\min\{ \beta_{\lambda}(x, y), 0 \} \\
&  & = \alpha(x, y)+\zeta(x, y)+2\min\{ \beta(x, y), 0 \} \\
&  & = \left\{\begin{array}{cl}
         4, & \textrm{in the case of} \ (2, 2, -2, -2, 0, 2), \\
         2, & \textrm{in the case of} \ (2, 2, -2, -1, 0, 0), \\
         2, & \textrm{in the case of} \ (2, 1, -1, -1, 0, 0), \\
         2, & \textrm{in the case of} \ (2, 0, 0, -1, 0, 0),
         \end{array}\right. \\
& & \delta_{\lambda}(x, y)+\varepsilon_{\lambda}(x, y)+2\min\{ \beta_{\lambda}(x, y), 0 \} \\
&  & = \delta(x, y)+\varepsilon(x, y)+2\min\{ \beta(x, y), 0 \} \\
&  & = \left\{\begin{array}{cl}
         -2, & \textrm{in the case of} \ (2, 2, -2, -2, 0, 2), \\
         -1, & \textrm{in the case of} \ (2, 2, -2, -1, 0, 0), \\
         -1, & \textrm{in the case of} \ (2, 1, -1, -1, 0, 0), \\
         -1, & \textrm{in the case of} \ (2, 0, 0, -1, 0, 0),
         \end{array}\right. \\
& & \alpha_{\lambda}(x, y)+\beta_{\lambda}(x, y)+\zeta_{\lambda}(x, y) \\
&  & = \alpha(x, y)+\beta(x, y)+\zeta(x, y) \\
&  & = \left\{\begin{array}{cl}
         6, & \textrm{in the case of} \ (2, 2, -2, -2, 0, 2), \\
         4, & \textrm{in the case of} \ (2, 2, -2, -1, 0, 0), \\
         3, & \textrm{in the case of} \ (2, 1, -1, -1, 0, 0), \\
         2, & \textrm{in the case of} \ (2, 0, 0, -1, 0, 0)
         \end{array}\right. \\
&  & \ge B,
\end{eqnarray*}
and
\begin{eqnarray*}
-\frac{\delta_{\lambda}(x, y)+\varepsilon_{\lambda}(x, y)+2\min\{ \beta_{\lambda}(x, y), 0 \}}{\alpha_{\lambda}(x, y)+\zeta_{\lambda}(x, y)+2\min\{ \beta_{\lambda}(x, y), 0 \}} = \frac{1}{2} = A.
\end{eqnarray*}
Unfortunately, the conditions of the theorem~\ref{th:3} are not satisfied in other cases.

Furthermore, since $\NN$ is discrete, Theorem~\ref{th:2} is sufficient without using Theorem~\ref{th:3}, however, the conditions of the theorem~\ref{th:2} are also not satisfied.
\end{remark}



\end{document}